\documentclass[12pt,a4paper]{article}

\usepackage[intlimits]{amsmath}
\usepackage{amsfonts,amssymb,amscd,amsthm,cite}
\usepackage{amstext}
\usepackage{mathtools}
\usepackage[usenames]{color}
\usepackage{graphicx}

\usepackage[all]{xy}
\input xy
\xyoption{all}
\usepackage{epstopdf}

\usepackage{geometry}
\usepackage[affil-it]{authblk}

\makeatletter
\def\@maketitle{%
  \newpage
  \null
  \vskip 2em%
  \begin{center}%
  \let \footnote \thanks
    {\Large\bfseries \@title \par}%
    \vskip 1.5em%
    {\normalsize
      \lineskip .5em%
      \begin{tabular}[t]{c}%
        \@author
      \end{tabular}\par}%
    \vskip 1em%
    {\normalsize \@date}%
  \end{center}%
  \par
  \vskip 1.5em}
\makeatother

\newtheorem{theorem}{Theorem}[section]
\newtheorem{corollary}[theorem]{Corollary}

\newtheorem{lemma}[theorem]{Lemma}
\newtheorem{proposition}[theorem]{Proposition}

\theoremstyle{definition}
\newtheorem{definition}[theorem]{Definition}

\theoremstyle{remark}

\newtheorem{remark}[theorem]{Remark}
\newtheorem*{remark*}{Remark}

\numberwithin{equation}{section}
\numberwithin{figure}{section}

\DeclareMathOperator{\codim}{codim}

\DeclareMathOperator{\sgn}{sgn}
\DeclareMathOperator{\ord}{ord}

\DeclareMathOperator{\rk}{rk}
\def\graph{\operatorname{graph}}

\def\lra{\longrightarrow}
\def\hra{\hookrightarrow}
\def\ra{\rightarrow}

\def\RR{\mathbb{R}}

\def\cD{{\mathcal{D}}}

\def\de{\delta}

\def\ov{\overline}

\def\wt{\widetilde}

\def\pa{\partial}

\newcommand{\abs}[1]{\left\lvert #1 \right\rvert}
\newcommand{\aabs}[1]{\lvert #1 \rvert}

\newcommand{\diff}[1]{\, d#1}

\def\const{\operatorname{const}}
\def\id{\operatorname{id}}

\newcommand{\FourierTransform}[2]{\mathcal{F}_{#1 \rightarrow #2}}
\newcommand{\InverseFourierTransform}[2]{\mathcal{F}^{-1}_{#1 \rightarrow #2}}
\newcommand{\NakedFourierTransform}{\mathcal{F}}
\newcommand{\NakedInverseFourierTransform}{\mathcal{F}^{-1}}

\def\eqdef{\stackrel{\mathclap{\text{def}}}{=}}

\def\Lag{\Lambda}

\def\XX{X\times X}
\def\MM{M\times M}
\def\LagXX{\Lag|_{\XX}}
\def\XXRRN{\XX \times \RR^N}
\def\XXRRzN{\XX \times \RR_0^N}
\def\MMRRN{\MM \times \RR^N}
\def\MMRRzN{\MM \times \RR_0^N}

\def\TM{T^*M}

\def\TzM{T^*_0M}
\def\TzMX{T^*_0M|_X}

\def\TMM{T^*(\MM)}
\def\TXX{T^*(\XX)}
\def\TMMXX{T^*(\MM)|_{\XX}}
\def\TzMM{T^*_0(\MM)}
\def\TzXX{T^*_0(\XX)}
\def\TzMMXX{T^*_0(\MM)|_{\XX}}

\def\phiXX{\phi|_{\XX}}
\def\aXX{a|_{\XX}}

\def\gammaPhi{\gamma_\phi}
\def\gammaPhiXX{\gamma_{\phi|_{\XX}}}

\def\CPhi{C_\phi}
\def\CPhiXX{C_{\phi|_{\XX}}}

\def\piXX{\pi_{\XX}}

\def\omXX{\omega_{\XX}}
\def\omMM{\omega_{\MM}}

\def\pt{\nu}
\newcommand{\Tpt}[2]{T_{#2}#1}

\begin{document}

\title{On traces of Fourier integral operators on submanifolds
\thanks{The work is supported by RFBR grant NN $16$-$01$-$00373$ A.}
}
\author{{P. Sipailo}\thanks{RUDN University, Russia}}

\maketitle

\begin{abstract}\noindent
Given a smooth embedding $i\colon X \hookrightarrow M$ of manifolds
and a Fourier integral operator $\Phi = \Phi(\Lambda)$ on $M$ associated with
a Lagrangian submanifold $\Lambda \subset T^*(X\times X) \setminus \{0\}$,
we consider its trace $i^!(\Lambda)$ on the submanifold $X$, i.e. the composition
$i^* \Phi i_*$, where $i^*$ and $i_*$ are the boundary and coboundary operators, respectively.
We establish the conditions under which the trace $i^!(\Phi)$ is also a Fourier integral operator,
and calculate its amplitude in canonical local coordinates.
\end{abstract}

\tableofcontents

\section{Introduction}
Given a smooth embedding $i\colon X\hra M$ of manifolds and an operator $A$ on $M$,
the \textit{trace of $A$} (see~\cite{NoSt1, NoSt2})
is an operator denoted by $i^!(A)$ on the submanifold $X$ given by the composition
$$
i^!(A) = i^* \, A \, i_*,
$$
where $i^*$ is the \textit{boundary operator}, i.e. the operator of restriction to $X$,
and $i_*$ is a dual \textit{coboundary operator}, which is defined in a dual manner
(more precisely, it takes a function on $X$ to a distribution on $M$ localized at $X$).
The concept of trace of operator  on a manifold  plays a central role in the so called
\textit{relative elliptic theory} (see~\cite{Ster9, NaSt10}), i.e. a theory associated with a pair of manifolds $(M,X)$. In particular
the trace operation is the main tool in studying the \textit{Sobolev problem},
which us a pseudodifferential problem with boundary conditions
posed on an embedded submanifold (see~\cite{Ster1,StSh53-1}).
It was quickly discovered that the trace operation behaves fairly well within the class of pseudodifferential operators (PDOs), namely the trace of a PDO is again a PDO.
But the things become more complicated if one deals with a wider class of operators.
For instance, given a $G$-operator (i.e. an operator associated with an action of a group $G$),
its trace turns out to be an operator of some very different nature:
in particular (under suitable conditions),
it is \textit{localized} at the set of fixed points of group action,
that is, it is compact outside any neighbourhood of this set (see~\cite{SaSt36}).
A similar situation occurs in studying manifolds with singularities (see~\cite{SaSt31}).

A closer analysis shows that in some situations such localized operators can be described as
Fourier integral operators (FIOs).
Besides that FIOs naturally appear in relative elliptic theory
as compositions with bundary/coboundary operators. 
Note also that the shift operators (building bricks for $G$-operators) are special cases
of quantized canonical transformations
(FIOs associated with graphs of canonical transformations),
and the latter operators arise in differential equations quite widely
(see also~\cite{SaSchSt10}, where $G$-operators
associated with groups of quantized canonical transformations were studied).

The above observations lead to the question whether the trace of a FIO on a submanifold
(or the trace of a quantized canonical transformation) is again a FIO on this submanifold.
It is easy to see that in general the answer is ``no'': the reason is that
given a trace of a FIO, the underlying Lagrangian submanifold in the cotangent bundle
may have singularities or boundary
(while classical FIOs are associated with smooth (immersed) Lagrangian submanifolds).
So the next natural question is to determine the conditions guaranteeing
that the trace of a FIO is a FIO again. This question is the main problem of the current paper.

Let us make some remarks on the present work. First, observe that
the boundary and coboundary operators are aslo special cases of FIOs,
so the subject of our study is in fact just a composition of three FIOs.
However, the classical theory of FIOs does not capture
such type of composition, since the Lagrangian submanifolds corresponding to $i^*$ and $i_*$
\textit{do intersect} the zero section in the factors of
$T^*M \times T^*X \simeq T^*X \times T^*M$.
Second, we note this paper is actually an extension of~\cite{Sip1}.
The differences are that now our traces are associated with general Lagrangian submanifolds
(while in~\cite{Sip1} these submanifolds were just isolated fibers of the cotangent bundle),
and now we deal with clean intersections rather than with transversal ones.

Finally, let us discuss the contents of the paper. First
(after recalling some basic facts about FIOs and traces of operators),
we introduce the notion of \textit{trace of Lagrangian submanifold},
associated with the embedding $i\colon X \hra M$. This is an operation, which
takes a Lagrangian submanifold $\Lag \subset \TMM$
to some submanifold of $i^!(\Lag) \subset \TXX$.
Then we formulate the main theorem of the paper, which states that the trace
of a FIO $\Phi = \Phi(\Lag)$ associated with a Lagrangian submanifold $\Lag$ under certain conditions
is a FIO associated with $i^!(\Lag)$ (and, in particular, $i^!(\Lag)$ is Lagrangian).
In other words, we describe the \textit{naturality} of the trace operation within the class of FIOs
in the sense of commutativity of the following diagram:
\begin{equation}\notag
\xymatrix{
\Lag
	\ar[r]^{}
	\ar[d]_{} &
i^!(\Lag)
	\ar[d]^{} \\
\Phi
	\ar[r]_{} &
i^!(\Phi)
}
\end{equation}
where the vertical arrows stand for the correspondence between Lagrangian manifolds
and associated FIOs. After that we calculate the amplitude of the resulting FIO
by means of canonical coordinates on $i^!(\Lag)$.
In the last section we discuss a special case of quantized canonical transformations.

The author is grateful to Prof. B.Yu.Sternin and to Prof. A.Yu.Savin for the support
and deep attention to his work.

\section{Preliminaries}

Throughout the paper, the subscript ``$0$'' means removing the zero section.
For example, $\RR^N_0 := \RR^N \setminus \{0\}$, $T_0^*M := T^*M\setminus \{0\}$, etc.

\subsection{Fourier integral operators}

Here we recall some basic facts from the theory of Fourier integral operators
(see~\cite{Hor3, Hor4, Treves2}).

Let $X$ be a smooth closed manifold of dimension $n$.
Denote by $(x,x')$ local coordinates on $\XX$,
and denote by $(x,p; \, x',p')$ the corresponding coordinates on $\TXX$.
Fix the symplectic form $\omXX = dx \wedge dp - dx' \wedge dp'$ on $\TXX$.

\begin{definition}\label{def:phaseFunction}
A function $\phi(x,x',\theta) \in C^\infty(\Gamma)$
defined on some open cone $\Gamma \subset \XXRRzN$
is called a \textit{phase function} if
\begin{itemize}
    \item[1)] it is real-valued and homogeneous of degree $1$ with respect to $\theta$-variables.
    \item[2)] its gradient $\pa_{x,x',\theta}\phi$ vanishes nowhere on $\Gamma$.
\end{itemize}
A phase function $\phi(x,x',\theta)$ is said to be \textit{clean},
if, in addition,
\begin{itemize}
    \item[3)] the set
\begin{equation}\notag
    \CPhi = \left\{\, (x,x',\theta) \in \Gamma \mid \pa_\theta \phi = 0 \,\right\}
\end{equation}
is a smooth conic manifold and the $N \times (2n + N)$-matrix
\begin{equation}\label{defeq:cleanMatrix}
\nabla(\pa_\theta\phi) =
\left(
  \begin{array}{ccc}
    \pa^2_{\theta x}\phi & \pa^2_{\theta x'}\phi & \pa^2_{\theta\theta}\phi \\
  \end{array}
\right)
\end{equation}
is of constant rank $\rk (\pa_{\theta x}\phi) = N - e$ on $\CPhi$ with
\begin{equation}\label{eq:excessCondition}
    e = \dim \CPhi - 2\dim X.
\end{equation}
\end{itemize}
\end{definition}
The set $\CPhi$ is called the \textit{critical set} of $\phi$,
and the number $e$ is called the \textit{excess} of $\phi$.
A phase function is \textit{nondegenerate} if its excess $e = 0$.

For brevity,
from now on we do not mention the domain cone $\Gamma$ explicitly and consider $\phi$ as a function
on the entire space $\XXRRzN$ (still tacitly assuming it is defined on some open cone).

\begin{definition}\label{def:assLagrangianManifold}
A smooth submanifold $\Lag \subset \TzXX$ is said to be \textit{associated}
with a phase function $\phi$
(we also say that $\Lag$ is \textit{parametrized} by $\phi$) if in some conic neighbourhood it is
defined as the range of the critical set $\CPhi$ under the map
\begin{equation}\notag
\gammaPhi\colon
\XXRRzN \longrightarrow \TzXX,
\quad
(x,x',\theta) \longmapsto (x,\pa_x\phi; \, x',-\pa_{x'}\phi),
\end{equation}
that is, $\gammaPhi(\CPhi)$ is an open subset of $\Lag$.
\end{definition}

It turns out (see, for example,~\cite{Hor3}) that given a clean phase function $\phi$,
the set $\gammaPhi(\CPhi)$ is an immersed conic submanifold in $\TzXX$,
and, moreover, it is Lagrangian with respect to the form $\omXX$.
Furthermore, in this case $\gammaPhi\colon \CPhi \ra \Lag$
is a fibration with fibers of dimension $e$ (and it is a local diffeomorphism when $e = 0$).
We call this map the \textit{parametrization} of $\Lag$ by $\phi$.

\begin{definition}\label{def:FIO}
A \textit{Fourier integral operator} $\Phi = \Phi(\Lag)$,
associated with a conic Lagrangian submanifold $\Lag \subset \TzXX$,
is a linear operator~$C^\infty_0(X) \ra \cD'(X)$
whose Schwartz kernel $K_\Phi \in \cD'(\XX)$ is locally of the form
\begin{equation}\label{eq:FIOKernel}
K_\Phi(x,x') = (2\pi)^{-(n+N)/2+e/2}
\int e^{i \phi(x,x',\theta)} \, a(x,x',\theta) \diff{\theta}.
\end{equation}
Here $\phi(x,x',\theta) \in C^\infty(\XXRRzN)$
is a clean phase function of excess $e$
parametrizing $\Lag$,
$a(x,x',\theta) \in S^{d+(n-N-e)/2}(\XXRRN)$
is an \textit{amplitude}
(a function from H\"ormander's symbol class),
the number $d$ is called the \textit{order} of $\Phi$
(we write $\ord \Phi = d$).
It is assumed that the support of $a$ is contained in the domain of $\phi$.
The integral is defined in the sense of distributions (as an oscillatory integral).
\end{definition}
Recall that the kernel~\eqref{eq:FIOKernel}
does not depend on the choice of $\phi$
modulo smooth functions (if one takes an appropriate amplitude),
provided that $\phi$ parametrizes $\Lag$
and the support of $a$ is sufficiently small.
Thus $\Phi(\Lag)$
does not depend on the way its Schwartz kernel is written down modulo smoothing operators,
but on the underlying Lagrangian submanifold $\Lag$
(in particular we can always assume that $\phi$ is nondegenerate,
since any Lagrangian manifold admits a parametrization by such a phase function).
See the details in~\cite{Hor4, Treves2}.

Now recall that distributions of the form~\eqref{eq:FIOKernel}
admit a special representation in the framework
of the theory of Maslov canonical operator~\cite{Mas1, NOSS1}).
Namely,
let the following collection of coordinate functions
\begin{equation}\label{eq:w}
w = (x_{I}, p_{\overline{I}};\,x'_{I'}, p'_{\overline{I}'}),
\end{equation}
where $I,I' \subset \{1,\dotsc,n\}$,
$\ov{I} = \{1,\dotsc,n\}\setminus I$,
$\ov{I}' = \{1,\dotsc,n\}\setminus I'$,
define a local coordinate system in some conic neighbourhood in $\Lag$
(\textit{canonical coordinates} on $\Lag$).
Then there is a smooth homogeneous function~$S(w)$ of degree $1$ on this neighbourhood
(a \textit{generating function} of $\Lag$)
such that $\Lag$ in the coordinates $(x,p;\,x',p')$ is defined by the equations
\begin{equation}\label{eq:LagCanonicalXX}
x_{\overline{I}} = -\frac{\partial S(w)}{\partial p_{\overline{I}}},\quad
p_{I} = \frac{\partial S(w)}{\partial x_{I}},\quad
x'_{\overline{I}'} = \frac{\partial S(w)}{\partial p'_{\overline{I}'}},\quad
p'_{I'} = -\frac{\partial S(w)}{\partial x'_{I'}},
\end{equation}
where $w \in \Lag$ is defined by~\eqref{eq:w}.
In this case the kernel~\eqref{eq:FIOKernel} modulo smooth functions can be expressed as
\begin{multline}\label{eq:FIOKernelCanonical}
K_\Phi(x,x') =
\InverseFourierTransform{p_{\overline{I}}}{x_{\overline{I}}}\,
\FourierTransform{p'_{\overline{I}'}}{x'_{\overline{I}'}} \,b(w) = \\ =
(2\pi)^{-(\lvert \overline{I} \rvert + \lvert \overline{I}' \rvert)/2}\iint
    e^{i S(w) + i p_{\overline{I}}x_{\overline{I}} - i p'_{\overline{I}'}x'_{\overline{I}'}} \,
    b(w) \diff{p_{\overline{I}}}\diff{p'_{\overline{I}'}},
\end{multline}
where
$b \in S^{d+(n-\lvert \overline{I} \rvert - \lvert \overline{I}' \rvert)/2}(\Lag)$,
and $\NakedFourierTransform$, $\NakedInverseFourierTransform$
stand for the direct and inverse Fourier transforms, respectively.
See also Proposition~\ref{prop:amplitude} below.

Lastly, we recall the notion of clean intersection of manifolds (see~\cite{Hor3}).
\begin{definition}\notag
Let $M_3$ be a manifold and let $M_1$ and $M_2$ be its submanifolds.
The intersection $M_1 \cap M_2 \subset M_3$ is said to be \textit{clean}
if it is a submanifold in $M_3$ and for any point $\pt \in M_1 \cap M_2$ we have
$$
    \Tpt{(M_1 \cap M_2)}{\pt} = \Tpt{M_1}{\pt} \cap \Tpt{M_2}{\pt}.
$$
\end{definition}

\subsection{Traces of operators and traces of Lagrangian manifolds}
Let $i\colon X \hookrightarrow M$ be a smooth embedding of closed manifolds.
Let $(x,y)$ be local coordinates on $M$
and let $X$ be defined in these coordinates by the equations
$X = \{y = 0\}$. Denote by $(x,y,p,q)$ the corresponding coordinates on $\TM$
and by
\begin{equation}\label{eq:coordsTMM}
    (x,y,p,q; \, x',y',p',q')
\end{equation}
the corresponding coordinates on $\TMM$.
Fix the symplectic form $\omMM$ on $T^*(M \times M)$ of the form
$$
    \omMM = dx \wedge dp + dy \wedge dq - dx' \wedge dp' - dy' \wedge dq'.
$$

The embedding $i$ induces two special operators, namely
the \textit{boundary operator} $i^*$
and the \textit{coboundary operator} $i_*$ (see~\cite{Ster1}).
The first one is an operator of restriction to the submanifold
and the second one acts in a dual manner. More explicitly,
in the above local coordinates these operators are defined as follows
\begin{equation}\label{defeq:cobb}
\begin{aligned}
& i^*\colon H^s(M) \lra H^{s-\nu/2}(X),     & \quad & u(x,y) \longmapsto u(x,0), \\
& i_*\colon H^{-s+\nu/2}(X) \lra H^{-s}(M), & \quad & u(x) \longmapsto u(x) \otimes \delta_X(y),
\end{aligned}
\end{equation}
where $\nu = \codim_M X$, and $\de_X(y)$ stands for the Dirac delta-function localized at $X$.
Both operators are continuous in the specified Sobolev spaces, provided that $s-\nu/2 > 0$.

Let $\Phi$ be an operator on the ambient manifold $M$.
\begin{definition}\notag 
The \textit{trace} $i^!(\Phi)$   %
of $\Phi$ on $X$ is the composition (see~\cite{NoSt1, NoSt2})
\begin{equation}\label{defeq:sled}
    i^!(\Phi) = i^* \, \Phi \, i_*.
\end{equation}
\end{definition}
\begin{remark}\label{remark:sledIsWellDefined}
Clearly, the trace $i^!(\Phi)$ is an operator on the submanifold $X$.
Note that the requirement $s-\nu/2 > 0$,
which limits the orders of the Sobolev spaces in~\eqref{defeq:cobb},
suggests that the composition~\eqref{defeq:sled} is not always well-defined.
Namely, $\Phi$ should be a continuous operator in the spaces $H^s(M) \ra H^{s-d}(M)$,
where
\begin{equation}\label{eq:sdnuSledIsDefined}
    s< -\nu/2, \quad s-d-\nu/2 > 0.
\end{equation}
In this case the trace $i^!(\Phi)$ is a continuous operator in the spaces
\begin{equation}\notag 
    i^!(\Phi)\colon H^{s+\nu/2}(X) \lra H^{s-d-\nu/2}(X).
\end{equation}
\end{remark}
Now let $\Lag$ be a submanifold in $\TzMM$.
\begin{definition}\notag 
The \textit{trace} $i^!(\Lag)$ of $\Lag$, associated to the embedding $i\colon X \hookrightarrow M$,
is the set
\begin{equation*}
i^!(\Lag) = \piXX(\LagXX),
\end{equation*}
where $\LagXX$ is the intersection
\begin{equation}\label{eq:LagXXisCap}
\LagXX = \Lag \, \cap \, \TMMXX
\end{equation}
and
\begin{equation}\notag 
\piXX \colon \TMMXX \longrightarrow \TXX
\end{equation}
stands for the projection induced by the embedding $i\times i\colon \XX \hra \MM$.
\end{definition}
Note that $i^!(\Lag)$ is a subset in $\TXX$.

\section{Traces of  Fourier integral operators}

\subsection{The main theorem}

Here we state the main result of the present paper.
\begin{theorem}\label{th:niceSled}
Let $\Phi = \Phi(\Lag)\colon H^s(M) \ra H^{s-d}(M)$
be a FIO associated with a Lagrangian submanifold~$\Lag \subset \TzMM$,
where $s$ and $d$ satisfy the inequalities~\eqref{eq:sdnuSledIsDefined}.
Let the following conditions hold:
\begin{enumerate}
\item[1)] the intersection~\eqref{eq:LagXXisCap} is clean;
\item[2)] one has $\Lag \cap N^*(\XX) = \emptyset$,
where $N^*(\XX)$ is the conormal bundle of $\XX \subset \MM$.
\end{enumerate}
Then $i^!(\Lag)$ is an immersed conic Lagrangian submanifold in $\TzXX$,
and $i^!(\Phi)$ is a FIO associated with it:
\begin{equation}\label{eq:niceSled}
i^!(\Phi(\Lag)) = \Phi(i^!(\Lag)).
\end{equation}
The order of $i^!(\Phi)$ is given by
\begin{equation}\label{eq:niceSledOrd}
\ord i^!(\Phi) = \ord \Phi - \dim X + \frac{1}{2}\codim X + \frac{1}{2}\dim \LagXX.
\end{equation}
\end{theorem}
\begin{proof}
First of all, the limitations on the action of $\Phi$ in Sobolev spaces
guarantee that $i^!(\Phi)$ is well-defined (see Remark~\ref{remark:sledIsWellDefined}).
Next, it clearly suffices to prove~\eqref{eq:niceSled} in local coordinates,
so we can restrict our attention to some small conic neighbourhood~$U \subset \TMM$
(with nonempty intersection with~$\TMMXX$) and assume that $\Lag$
is associated with some nondegenerate phase function in this neighbourhood.
For brevity, from now on we identify
all the manifolds under consideration with their neighbourhoods corresponding to $U$:
for example, we write $\TMM$ instead of $U$, $\Lag$ instead of $\Lag \cap U$, etc.
Let us also assume that $U$ is equipped with local coordinates of the form~\eqref{eq:coordsTMM}.

Under the above assumptions, we may consider $\Phi$ as an integral operator
with Schwartz kernel
\begin{equation}\label{eq:kernelPhi}
K_\Phi(x,y,x',y') = (2\pi)^{-(\dim M+N)/2}
\int
    e^{i\phi(x,y,x',y',\theta)}\,
    a(x,y,x',y',\theta) \,
\diff{\theta},
\end{equation}
where $\phi \in C^\infty(\MMRRzN)$
is a nondegenerate phase function, which parameterizes $\Lag$,
$a \in S^{d+(\dim M-N)/2}(\MMRRN)$ is an amplitude, and $d = \ord \Phi$.
Next, by a straightforward computation, we see that the trace~$i^!(\Phi(\Lag))$
is an integral operator with Schwartz kernel
\begin{equation}\label{eq:kernelSledPhi}
K_{i^!(\Phi)}(x,x') = (2\pi)^{-(\dim M+N)/2}
\int
    e^{i\phiXX(x,x',\theta)}\,
    \aXX(x,x',\theta)\,
\diff{\theta},
\end{equation}
where
$$
\phiXX(x,x',\theta) = \phi(x,0,x',0,\theta),
\quad
\aXX(x,x',\theta) = a(x,0,x',0,\theta).
$$
Note that
$\aXX$ is an amplitude of the same order as~$a$,
i.e. \begin{equation}\label{eq:aXspace}
    a(x,0,x',0,\theta) \in S^{d+(\dim M-N)/2}(\XXRRN).
\end{equation}
Thus it is enough to prove that $\phiXX$ is a clean phase function associated with $i^!(\Lag)$.
This will imply~\eqref{eq:niceSled}.

\textbf{Step 1.}
\textit{Parametrization of~$i^!(L)$}.
Let us show that~$\phiXX$ parameterizes $i^!(\Lag)$
in the sense of Definition~\ref{def:assLagrangianManifold}.
We start from recalling the properties of $\phi$.

Since $\phi$ is nondegenerate, its critical set
\begin{equation}\notag
\CPhi =
    \left\{\,
        (x,y,x',y',\theta) \mid \pa_\theta \phi = 0
    \,\right\}
\subset \MMRRzN
\end{equation}
is a smooth manifold of dimension~$\dim \CPhi = 2\dim M$,
and~$\Lag$ is the range of this submanifold under the map
\begin{equation}\notag
\begin{array}{rcl}
  \gammaPhi\colon \MMRRzN & \longrightarrow & \TzMM, \vspace{1ex} \\
  (x,y,x',y',\theta) & \longmapsto &
    (x,y,\pa_x\phi, \pa_y\phi; \, x',y',-\pa_{x'}\phi, -\pa_{y'}\phi).
\end{array}
\end{equation}
Furthermore, we can assume that the restriction of~$\gammaPhi$ to~$\CPhi$ defines a diffeomorphism
\begin{equation}\label{eq:gammaPhiDiff}
\gammaPhi \colon \CPhi \xrightarrow{~\simeq~} \Lag.
\end{equation}
This situation transfers to $\phiXX$ as follows.
First, the critical set of~$\phiXX$ is of the form
\begin{equation}\notag
\CPhiXX =
    \left\{\,
        (x,x',\theta) \mid \pa_\theta (\phiXX) = 0
    \,\right\}
\subset \XXRRzN.
\end{equation}
Note that
\begin{equation}\label{eq:CPhiXXcapXXRRzN}
\CPhiXX = \CPhi \cap \XXRRzN,
\end{equation}
and
\begin{equation}\label{eq:LagXXgammaPhiCPhiXX}
\LagXX = \gammaPhi(\CPhiXX).
\end{equation}
Now, since the intersection~\eqref{eq:LagXXisCap} is clean,
$\LagXX$ is a submanifold in $\Lag$;
hence, reverting the diffeomorphism~\eqref{eq:gammaPhiDiff},
we deduce that $\CPhiXX$ is a submanifold in $\CPhi$.
(More precisely, $\gammaPhi$ restricted to $\CPhiXX$ defines a diffeomorphism
$\CPhiXX \rightarrow \LagXX$.)

Second, the parametrization map corresponding to $\phiXX$ is
\begin{equation}\notag
\begin{array}{rcl}
  \gammaPhiXX\colon \XXRRzN & \longrightarrow & \TzXX, \vspace{1ex} \\
  (x,x',\theta) & \longmapsto & (x,\pa_x (\phiXX); \, x', -\pa_{x'} (\phiXX)).
\end{array}
\end{equation}
Evidently, it is a composition
of the restriction $\gammaPhi$ to $\XXRRzN$ and the projection $\piXX$.
Together with~\eqref{eq:CPhiXXcapXXRRzN} and~\eqref{eq:LagXXgammaPhiCPhiXX}
this implies the identity
\begin{equation}\label{eq:i^!(L)isParametrized}
    i^!(\Lag) = \gammaPhiXX(\CPhiXX).
\end{equation}
Thus $\phiXX$ is associated with $i^!(\Lag)$
in the sense of Definition~\ref{def:assLagrangianManifold}, as desired.

Summarising, we get the following commutative diagram:
\begin{equation}\label{eq:diag}
\xymatrix{
  \CPhiXX \ar[dr]_{\gammaPhiXX} \ar[r]^{\gammaPhi}
                & \LagXX \ar[d]^{\piXX}  \\
                & i^!(\Lag)             }
\end{equation}
where the $\gammaPhi$ is a diffeomorphism, and the $\piXX$
is a smooth map of constant rank.
\begin{corollary}\label{cor:i^!(Lag)isSubmanifoldInTzXX}
The set $i^!(\Lag)$ is an immersed conic submanifold in $\TzXX$.
\end{corollary}
\begin{proof}
From the commutativity of the diagram~\eqref{eq:diag} we see that $\gammaPhiXX$
is a smooth map of constant rank (equal to the rank of $\piXX$), hence
its image $i^!(\Lag)$ is an immersed submanifold in $\TXX$.
The fact that it does not intersect the zero section of $\TXX$ follows
from the hypothesis 2) of the current theorem.
Indeed, let $\{0\} \subset \TXX$ be the zero section. Then its preimage
$[\piXX]^{-1}(\{0\})$ under the projection $\piXX$
is precisely the conormal bundle of $\XX \hra \MM$,
$$
[\piXX]^{-1}(\{0\}) = N^*(\XX).
$$
So $\Lag \cap N^*(\XX) = \emptyset$ implies $i^!(\Lag) \cap \{0\} = \emptyset$.
Corollary~\ref{cor:i^!(Lag)isSubmanifoldInTzXX} follows.
\end{proof}

On the next step we study the critical set of $\phiXX$.

\textbf{Step 2.}
\textit{Properties of $\CPhiXX$.}
Here we show that the intersection~\eqref{eq:LagXXisCap}
transfers to the ``parameter space'' $\MMRRzN$ and remains clean.
We start from the subspace~$\TzMMXX$.
\begin{lemma}\notag
One has
\begin{equation}\label{eq:TzXXisParametrized}
    \gammaPhi(\XXRRzN) \subseteq \TzMMXX.
\end{equation}
\end{lemma}
\begin{proof}
This follows immediately from a direct computation.
\end{proof}
Now we can describe $\CPhiXX$.
\begin{lemma}\label{lemma:cleanessInParameterSpace}
The set~$\CPhiXX$ is a submanifold in $\MMRRzN$,
and for any point $\pt \in \CPhiXX$ we have
\begin{equation}\label{eq:TCritPhiXXcap}
    \Tpt{\CPhiXX}{\pt} = \Tpt{\CPhi}{\pt} \, \cap \, \Tpt{(\XXRRzN)}{\pt}.
\end{equation}
\end{lemma}
\begin{remark}
In other words, Lemma~\ref{lemma:cleanessInParameterSpace} claims
that the intersection~\eqref{eq:CPhiXXcapXXRRzN} is clean.
\end{remark}
\begin{proof}
We have already seen that $\CPhiXX$ is a submanifold in $\CPhi$.
Since $\CPhi$ is a submanifold in $\MMRRzN$,
it follows that $\CPhiXX$ is a submanifold in $\MMRRzN$, as claimed.
It remains to prove~\eqref{eq:TCritPhiXXcap}.

Let $\pt \in \CPhiXX$ be fixed.
Then~\eqref{eq:CPhiXXcapXXRRzN} implies the inclusion
$$
\Tpt{\CPhiXX}{\pt} \subseteq \Tpt{\CPhi}{\pt} \cap \Tpt{(\XXRRzN)}{\pt},
$$
hence, in order to establish~\eqref{eq:TCritPhiXXcap}, it suffices to obtain the inverse inclusion.
Now consider the linear map
$$
    d\gammaPhi\colon \Tpt{(\MMRRzN)}{\pt} \longrightarrow \Tpt{(\TzMM)}{\gammaPhi(\pt)},
$$
induced by $\gammaPhi$.
Its restriction to $\Tpt{\CPhi}{\pt} \subset \Tpt{(\MMRRzN)}{\pt}$ defines
an isomorphism of vector spaces
\begin{equation}\label{eq:(dgammaPhi)Iso}
d\gammaPhi\colon \Tpt{\CPhi}{\pt} \xrightarrow{~\simeq~} \Tpt{\Lag}{\gammaPhi(\pt)}.
\end{equation}
Consequently, the desired inclusion is equivalent to the following:
\begin{equation}\notag
d\gammaPhi(\Tpt{\CPhi}{\pt} \cap \Tpt{(\XXRRzN)}{\pt}) \subseteq d\gammaPhi(\Tpt{\CPhiXX}{\pt}).
\end{equation}
On the other hand, by~\eqref{eq:LagXXgammaPhiCPhiXX} we have
$$
d\gammaPhi(\Tpt{\CPhiXX}{\pt}) = \Tpt{(\LagXX)}{\gammaPhi(\pt)}.
$$
So it suffices to check that we have the inclusion
\begin{equation}\label{eq:dgammaPhi(CritPhiTXXRRzN)}
d\gammaPhi(\Tpt{\CPhi}{\pt} \cap \Tpt{(\XXRRzN)}{\pt}) \subseteq \Tpt{(\LagXX)}{\gammaPhi(\pt)}.
\end{equation}
Now note that~\eqref{eq:TzXXisParametrized} and~\eqref{eq:(dgammaPhi)Iso} imply
\begin{equation}\notag
d\gammaPhi(\Tpt{(\XXRRzN)}{\pt}) \subseteq \Tpt{(\TzMMXX)}{\gammaPhi(\pt)},
\quad d\gammaPhi(\Tpt{\CPhi}{\pt}) = \Tpt{\Lag}{\gammaPhi(\pt)}.
\end{equation}
Therefore, we have
\begin{multline}\notag
d\gammaPhi(\Tpt{\CPhi}{\pt} \cap \Tpt{(\XXRRzN)}{\pt}) \subseteq
d\gammaPhi(\Tpt{\CPhi}{\pt}) \cap d\gammaPhi(\Tpt{(\XXRRzN)}{\pt}) \subseteq \\ \subseteq
\Tpt{\Lag}{\gammaPhi(\pt)} \cap \Tpt{(\TzMMXX)}{\gammaPhi(\pt)} =
\Tpt{(\LagXX)}{\gammaPhi(\pt)},
\end{multline}
where the last equation holds because the intersection~\eqref{eq:LagXXisCap} is clean.
It follows that~\eqref{eq:dgammaPhi(CritPhiTXXRRzN)} holds,
and so Lemma~\ref{lemma:cleanessInParameterSpace} is proved.
\end{proof}

Let us now study $\phiXX$.

\textbf{Step 3.}
\textit{Properties of~$\phiXX$.}
\begin{lemma}\label{lemma:phiXXisClean}
The function~$\phiXX$ is a clean phase function with excess
\begin{equation}\label{eq:niceExcess}
    e = \dim \LagXX - 2\dim X.
\end{equation}
\end{lemma}
\begin{proof}
Let us check that $\phiXX$ meets the requirements listed in Definition~\ref{def:phaseFunction}.

1)~$\phiXX$ is real-valued and homogeneous of degree $1$ with respect to $\theta$-variables.
This is obvious.

2) The gradient $\pa_{x,x',\theta} (\phiXX)$ vanishes nowhere on $\CPhiXX$.
Indeed, otherwise the set~\eqref{eq:i^!(L)isParametrized}
would have a nonempty intersection with the zero section $\{0\} \subset \TXX$
contradicting Corollary~\ref{cor:i^!(Lag)isSubmanifoldInTzXX}.

3) Lemma~\ref{lemma:cleanessInParameterSpace} implies that $\phiXX$ has an excess
in the sense of Definition~\ref{def:phaseFunction}.
Indeed, a direct computation shows that~\eqref{eq:TCritPhiXXcap} leads to the equality
$$
    \Tpt{\CPhiXX}{\pt} = \operatorname{Ker}\Psi(\pt), \qquad \forall \pt \in \CPhiXX,
$$
where $\Psi$ is a linear operator given by the $N\times(2\dim X + N)$ matrix
\begin{equation}\label{eq:phiXXnabla}
\left(
  \begin{array}{ccc}
    \pa^2_{\theta x}(\phiXX) & \pa^2_{\theta x'}(\phiXX) & \pa^2_{\theta\theta}(\phiXX) \\
  \end{array}
\right).
\end{equation}
Consequently,
$$
    \dim \CPhiXX = 2\dim X +N - \operatorname{rk}\Psi.
$$
Therefore
\begin{equation}\label{eq:niceRk}
    \operatorname{rk}\Psi = N - e, \quad e = \dim \CPhiXX - 2\dim X.
\end{equation}
On the other hand, note that the matrix~\eqref{eq:phiXXnabla} is nothing
but the matrix~$\nabla(\partial_\theta\phiXX)$ (see~\eqref{defeq:cleanMatrix}),
so~\eqref{eq:niceRk} means that the number~$e$ is the excess of~$\phiXX$
(compare~\eqref{eq:niceRk} and~\eqref{eq:excessCondition}).
Finally, by virtue of the diffeomorphism $\CPhiXX \simeq \LagXX$
(see~\eqref{eq:diag}) we have $\dim \CPhi = \dim \LagXX$,
so~\eqref{eq:niceExcess} holds.

The properties 1)--2) mean that $\phiXX$ is a phase function,
and the property 3) means it is clean with excess~\eqref{eq:niceExcess}.
Lemma~\ref{lemma:phiXXisClean} is proved.
\end{proof}
\textbf{Step 4.}
\textit{Conclusion.}
The above arguments show that the manifold~$i^!(\Lag)$
is associated with a clean phase function $\phiXX$ of excess~\eqref{eq:niceExcess},
and the map
\begin{equation}\label{eq:phiXXSubmersion}
\gammaPhiXX\colon \CPhiXX \lra i^!(\Lag)
\end{equation}
defines the corresponding parametrization.
It follows that $i^!(\Lag)$ is an immersed Lagrangian submanifold in~$\TzXX$,
and the expression~\eqref{eq:kernelSledPhi} defines a kernel of a FIO
associated with $i^!(\Lag)$.
The formula~\eqref{eq:niceSledOrd}
follows directly from~\eqref{eq:aXspace} and~\eqref{eq:niceExcess}
(see Definition~\ref{def:FIO}).

The proof of Theorem~\ref{th:niceSled} is complete.
\end{proof}

\subsection{Calculation of amplitude}

Now we refine Theorem~\ref{th:niceSled} by representing the
kernel~\eqref{eq:kernelSledPhi} in the form of~\eqref{eq:FIOKernelCanonical}
for some amplitude $b(w)$ on $i^!(\Lag)$.

At first we need to make some preparations.
Assume that a local conic chart $U \subset \TMM$ with coordinates~\eqref{eq:coordsTMM}
is chosen\footnote{
As before, we write $\TMM$ instead of $U$, $\Lag$ instead of $\Lag \cap U$, etc.},
the conditions of Theorem~\ref{th:niceSled} are fulfilled, and
the kernel $K_\Phi$ of $\Phi$ is of the form~\eqref{eq:kernelPhi}.
Then $i^!(\Lag)$ is a Lagrangian submanifold,
and we have two natural ways to describe it.

1. On the one hand, since $i^!(\Lag)$ is Lagrangian,
there is a collection of coordinate functions
\begin{equation}\label{eq:canonicalCollectionPoint}
w = (x_{I}, p_{\overline{I}};\,x'_{I'}, p'_{\overline{I}'}),
\end{equation}
which defines a coordinate system on $i^!(\Lag)$.
Let such a collection be fixed,
and let $S(w)$ be the corresponding generating function of $i^!(\Lag)$.
Then $i^!(\Lag)$ is defined by the equations~\eqref{eq:LagCanonicalXX}.

2. On the other hand, according to the proof of Theorem~\ref{th:niceSled},
$i^!(\Lag)$ is associated with a clean phase function $\phiXX$ of excess $e$
(the latter is defined by~\eqref{eq:niceExcess}).
It follows that the parametrization~\eqref{eq:phiXXSubmersion}
is a fibration whose fibers
$$
F_w \;\eqdef\; [\gammaPhiXX]^{-1}(w) \,\cap\, \CPhiXX, \quad w \in i^!(\Lag),
$$
are smooth $e$-dimensional manifolds.
Moreover, it can be shown
(see the procedure of \textit{elimination of excess} described in~\cite{Treves2})
that there is
(possibly after a linear transformation of $\theta$-variables
and provided that the neighbourhood
$U$ is sufficiently small)
a splitting 
$$
\theta = (\theta',\theta''), \quad \theta' \in \RR^{N-e}, \, \theta'' \in \RR^e,
$$
such that the variables $\theta''$
define local coordinates in the fibers of~\eqref{eq:phiXXSubmersion},
and $\theta' \neq 0$ for all $(x,x',\theta) \in \CPhiXX$.

The next proposition connects two different expressions of the kernel $K_{i^!(\Phi)}$
corresponding to the two descriptions of $i^!(\Lag)$ given above.
Here for simplicity we assume that the amplitude $a$ in~\eqref{eq:kernelPhi}
is a \textit{classical symbol}, i.e. it admits an asymptotic expansion in decreasing orders of
homogeneity (see~\cite{Shu1}). By $a_0$ we denote the leading term for $a$.
\begin{proposition}\label{prop:amplitude}
Under the conditions of Theorem~\ref{th:niceSled}
the kernel $K_{i^!(\Phi)}$ modulo smooth functions
is of the form~\eqref{eq:FIOKernelCanonical},
where $b$ is a classical symbol and its leading term is given by
\begin{multline}\notag
b_0(w) =
(2\pi)^{-(\dim M+N-e)/2}
\int_{F_w}
e^{\frac{i\pi}{4} \sgn H_{w,\theta''}(\wt{x}_{\ov{I}},\wt{x}'_{\ov{I}'},\wt{\theta}') }
\,\times \\ \times
\abs{\det H_{w,\theta''}(\wt{x}_{\ov{I}},\wt{x}'_{\ov{I}'},\wt{\theta}')}^{-1/2}
a_0|_{\XX}(x_I,\wt{x}_{\ov{I}},x'_{I'},\wt{x}'_{\ov{I}'},\wt{\theta}',\theta'')
\diff{\theta''},
\end{multline}
where $w \in i^!(\Lag)$ is given by~\eqref{eq:canonicalCollectionPoint},
$e = \dim \LagXX - 2\dim X$,
$H_{w,\theta''}$ is the Hessian matrix of
$\phiXX(x,x',\theta)$ with respect to the variables
$(x_{\ov{I}},x_{\ov{I}'},\theta')$, and the point
$(\wt{x}_{\ov{I}},\wt{x}'_{\ov{I}'},\wt{\theta}')$ is determined by $w$ and $\theta''$
via the equation
\begin{equation}\notag
\gammaPhiXX(x_I,\wt{x}_{\ov{I}},x'_{I'},\wt{x}'_{\ov{I}'},\wt{\theta}',\theta'')
= w.
\end{equation}
\end{proposition}
\begin{proof}
We want to find an amplitude $b$ on $i^!(\Lag)$
such that the expressions~\eqref{eq:kernelSledPhi} and~\eqref{eq:FIOKernelCanonical}
define the same distribution modulo smooth functions.
Applying the composition of Fourier transforms
$\FourierTransform{x_{\overline{I}}}{p_{\overline{I}}}\,
\InverseFourierTransform{x'_{\overline{I}'}}{p'_{\overline{I}'}}$
to both of these expressions, we get
\begin{multline}\label{eq:bigAmplitudeIntegral}
b(w) =
(2\pi)^{-(\dim M+N)/2-(\lvert \overline{I} \rvert + \lvert \overline{I}' \rvert)/2}
\int
    e^{i\left[
        \phiXX(x,x',\theta)-S(w)
        -p_{\overline{I}}x_{\overline{I}} + p'_{\overline{I}'}x'_{\overline{I}'}
    \right]}\, \times \\ \times
    \aXX(x,x',\theta)\,
\diff{x_{\overline{I}}}\diff{x'_{\overline{I}'}}
\diff{\theta}.
\end{multline}
Now the rest of the proof is a computation of the integral~\eqref{eq:bigAmplitudeIntegral}
via the method of stationary phase. We refer the reader to the proof of Proposition $25.1.5'$
in~\cite{Hor4}, where an analogous integral was considered,
and only sketch out some basic points of this computation.

\textbf{Step 1.}
\textit{Determining the stationary points.}
Since the integral~\eqref{eq:bigAmplitudeIntegral} depends on the parameter $w \in i^!(\Lag)$
given by the coordinate functions~\eqref{eq:canonicalCollectionPoint},
let us assume that this parameter is fixed.
A straightforward calculation shows that a value of the collection
$(x_{\overline{I}},x'_{\overline{I}'},\theta)$ defines a stationary point
for the integral~\eqref{eq:bigAmplitudeIntegral} if the corresponding value
of the collection $(x,x',\theta)$ defines a point in $F_w$.
Thus we may assume that the integration is being performed over
some neighbourhood of the set 
\begin{equation}\notag
\wt{F}_w \; \eqdef \;
\{\,
    (x_{\overline{I}},x'_{\overline{I}'},\theta) \mid
    (x,x',\theta) \in F_w
\,\}.
\end{equation}

\textbf{Step 2.}
\textit{Reducing to a repeated integral.}
Using the splitting $\theta = (\theta',\theta'')$ and the fact that $\theta''$
define local coordinates in $F_w$,
we can rewrite the integral~\eqref{eq:bigAmplitudeIntegral} in the form
\begin{equation}\label{eq:bigAmplitudeIntegralOuter}
b(w) = \int_{F_w} c(w,\theta'')\diff{\theta''},
\end{equation}
where $c(w,\theta'')$ is given by
\begin{multline}\label{eq:bigAmplitudeIntegralInner}
c(w,\theta'') =
(2\pi)^{-(\dim M+N)/2-(\lvert \overline{I} \rvert + \lvert \overline{I}' \rvert)/2}
\int_{V_{\theta''}}
    e^{i\left[
        \phiXX(x,x',\theta',\theta'')-S(w)
        -p_{\overline{I}}x_{\overline{I}} + p'_{\overline{I}'}x'_{\overline{I}'}
    \right]}\, \times \\ \times
    \aXX(x,x',\theta',\theta'')
\diff{x_{\overline{I}}}\diff{x'_{\overline{I}'}}
\diff{\theta'},
\end{multline}
where ${V_{\theta''}}$ is the set of all $(x_{\overline{I}},x'_{\overline{I}'},\theta')$
such that
$(x_{\overline{I}},x'_{\overline{I}'},\theta)$ lie in a neighbourhood 
of $\wt{F}_w$.
We claim that the integral~\eqref{eq:bigAmplitudeIntegralOuter}
is actually over a bounded domain, therefore it converges.
Indeed, since $\theta' \neq 0$ for all $(x,x',\theta) \in F_w \subset \CPhiXX$,
the same remains true for all points in some conic neighbourhood $W$ of $F_w$.
But then the values of $\abs{\theta'}$ for $(x,x',\theta) \in W$ can not be arbitrary small;
since $W$ is conic it follows that $\abs{\theta''}$
can not be arbitrary large.
This means that the values of $\abs{\theta''}$ are bounded for all points in $F_w$, as claimed.

\textbf{Step 3.}
\textit{Calculating $c(w,\theta'')$.}
Now let us consider the integral~\eqref{eq:bigAmplitudeIntegralInner}
as depending on the parameter $\theta''$ and assume that the value $\theta'' = \const$
is fixed.
The idea is to apply to~\eqref{eq:bigAmplitudeIntegralInner} the method of stationary phase.
We make the following observations.

1) The point $(x_{\ov{I}},x'_{\ov{I}'},\theta')$
is stationary for the integral~\eqref{eq:bigAmplitudeIntegralInner}
if $(x,x',\theta) \in F_w \cap \{\theta'' = \const\}$.
It follows that this point is unique, provided that the neighbourhood $U$ is sufficiently small.

2) The phase function of the integral~\eqref{eq:bigAmplitudeIntegralInner} is given by
\begin{equation}\label{eq:phase}
(x_{\overline{I}},x'_{\overline{I}'},\theta'')
\longmapsto
\phiXX(x,x',\theta',\theta'')-S(w)
        -p_{\overline{I}}x_{\overline{I}} + p'_{\overline{I}'}x'_{\overline{I}'}.
\end{equation}
Its Hessian matrix is of the form
\begin{equation}\notag
H_{w,\theta''}(x_{\ov{I}},x'_{\ov{I}'},\theta') =
\left(
  \begin{array}{ccc}
    \pa^2_{x_{\ov{I}}\,x_{\ov{I}}\,}(\phiXX) &
    \pa^2_{x_{\ov{I}}\,x'_{\ov{I}'}\,}(\phiXX) &
    \pa^2_{x_{\ov{I}}\,\theta'\,}(\phiXX)
        \vspace{1ex}\\
    \pa^2_{x'_{\ov{I}'}\,x_{\ov{I}}\,}(\phiXX) &
    \pa^2_{x'_{\ov{I}'}\,x'_{\ov{I}'}\,}(\phiXX) &
    \pa^2_{x'_{\ov{I}'}\,\theta'\,}(\phiXX)
        \vspace{1ex}\\
    \pa^2_{\theta'\,x_{\ov{I}}\,}(\phiXX) &
    \pa^2_{\theta'\,x'_{\ov{I}'}\,}(\phiXX) &
    \pa^2_{\theta'\,\theta'\,}(\phiXX)
  \end{array}
\right)
\end{equation}
The next lemma shows that this matrix is nondegenerate at the stationary point.
\begin{lemma}
The matrix $H_{w,\theta''}(x_{\ov{I}},x'_{\ov{I}'},\theta')$ is nondegenerate
for all $(x_{\ov{I}},x'_{\ov{I}'},\theta')$ such that
$(x,x',\theta) \in F_w \cap \{\theta'' = \const\}$.
\end{lemma}
\begin{proof}
Consider the composition
\begin{equation}\notag
\CPhiXX \cap \{\, \theta'' = \const \,\}
\lra
i^!(\Lag)
\lra
\RR^{\aabs{I}}_{x_{I}} \times
\RR^{\aabs{\overline{I}}}_{p_{\overline{I}}} \times
\RR^{\aabs{I'}}_{x'_{I'}} \times
\RR^{\aabs{\overline{I}'}}_{p'_{\overline{I}'}}
\end{equation}
given by
\begin{multline}\notag
(x_{I},x_{\overline{I}},x'_{I'},x'_{\overline{I}'},\theta)
\longmapsto
(x,\pa_x (\phiXX); x',-\pa_{x'}(\phiXX))
\longmapsto \\
(x_{I},\pa_{x_{\overline{I}}}(\phiXX); \, x'_{I'},-\pa_{x'_{\overline{I}'}}(\phiXX))
\end{multline}
(the first arrow is the parametrization $\gammaPhiXX$,
and the second arrow is the coordinate map).
By construction this composition is a diffeomorphism onto its image.
We complete it to the map
\begin{equation}\notag
(x_{I},x_{\overline{I}},x'_{I'},x'_{\overline{I}'},\theta)
\longmapsto
(x_{I},\pa_{x_{\overline{I}}}(\phiXX); \, x'_{I'},-\pa_{x'_{\overline{I}'}}(\phiXX);
\, \pa_{\theta'}(\phiXX)).
\end{equation}
Since $\pa_{\theta'}(\phiXX) = 0$ on $\CPhiXX$, it follows that this map has surjective differential
for all $(x,x',\theta)$ such that
$(x,x',\theta) \in \CPhiXX \cap \{\, \theta'' = \const \,\}$.
Consequently, the map
$$
(x_{\overline{I}},x'_{\overline{I}'},\theta')
\longmapsto
(\pa_{x_{\overline{I}}}(\phiXX), \,\pa_{x'_{\overline{I}'}}(\phiXX), \,\pa_{\theta'}(\phiXX))
$$
has surjective differential for all $(x_{\overline{I}},x'_{\overline{I}'},\theta')$ such that
$(x,x',\theta) \in F_w \cap \{\theta'' = \const\}$.
Therefore its Jacobian matrix is nondegenerate at such points.
But this Jacobian matrix is equal to $H_{w,\theta''}(x_{\ov{I}},x'_{\ov{I}'},\theta')$,
and this proves the lemma.
\end{proof}

3) Finally, it is easy to see that the value of the function~\eqref{eq:phase}
is zero at the stationary point.

Now, applying the method of stationary phase to $c(w,\theta'')$ and substituting it
into~\eqref{eq:bigAmplitudeIntegralOuter}, we obtain the desired formula for $b(w)$.

The proof of Proposition~\ref{prop:amplitude} is complete.
\end{proof}

\subsection{Application to quantized canonical transformations}
In this section we apply Theorem~\ref{th:niceSled} to
{quantized canonical transformations}.

Let us recall some basic definitions.
Let $g\colon \TzM \ra \TzM$ be a homogenous canonical transformation
(i.e. a conic diffeomorphism preserving the symplectic form $\omMM$).
Then its graph
$$
\graph g = \{\, (g(w'),w') \,\} \subset \TzM \times \TzM.
$$
is a Lagrangian submanifold in $\TzMM$.
A FIO $\Phi = \Phi(\graph g)$ associated with $\graph g$ is called
a \textit{quantized canonical transformation}. One of the main features of these operators
is that they are bounded in the whole scale of Sobolev spaces.
Namely, $\Phi = \Phi(\graph g)$ acts continuously in the spaces
$$
\Phi\colon H^s(M) \lra H^{s-\ord \Phi}(M) \quad \forall s.
$$

The next corollary is a particular case of Theorem~\ref{th:niceSled}.
\begin{corollary}\label{cor:qnaturality}
Let~$\Phi = \Phi(\graph g)$ be a quantized canonical transformation
of order $\ord \Phi <-\codim X$.
Let the canonical transformation $g$ satisfy the following conditions:
\begin{enumerate}
\item[1)] the intersection $\TzMX \cap g(\TzMX) \subset \TzM$ is clean;
\item[2)] one has $N^*_0X \cap g(N^*_0X) = \emptyset$,
where $N^*X$ is the conormal bundle of $X \subset M$.
\end{enumerate}
Then $i^!(\graph g)$ is an immersed Lagrangian submanifold in $\TzXX$,
and $i^!(\Phi)$ is a FIO associated with it:
\begin{equation}\notag
    i^!(\Phi(\graph g)) = \Phi(i^!(\graph g)).
\end{equation}
\end{corollary}
\begin{proof}
The requirement $\ord \Phi <-\codim X$ guarantees that the trace $i^!(\Phi)$ is well-defined.
Let us show that the conditions 1) and 2) in Corollary~\ref{cor:qnaturality}
imply the conditions 1) and 2) in Theorem~\ref{th:niceSled}.

\textbf{Step 1.} \textit{Condition 1).}
We are going to check that the intersection
\begin{equation}\notag
(\graph g)|_{\XX} = \graph g \, \cap \, \TzMMXX
\end{equation}
is clean. To simplify the notation let us prove this fact
in a slightly more abstract setting.
\begin{lemma}\label{lemma:graphCleanIntersection}
Let~$f\colon Y \ra Y$ be a diffeomorphism between smooth manifolds, and let $Z \subset Y$
be a submanifold.
If the intersection $Z \cap f(Z) \subset Y$ is clean
then so is the intersection $\graph f \cap Z\times Z \subset Y \times Y$.
\end{lemma}
\begin{proof}
1) First, let us show that the set $\graph f \cap Z \times Z$ is a submanifold in $Y \times Y$.
Indeed, denote by $(f,\id)$ the map
\begin{equation}\notag
	(f,\id)\colon Y \lra Y \times Y, \quad \pt \longmapsto (f(\pt),\pt).
\end{equation}
This map is clearly a diffeomorphism $Y \ra \graph f$, and, moreover,
\begin{equation}\label{eq:graphfZZRepresent}
\graph f \cap Z \times Z = (f,\id)\,[f^{-1}(Z \cap f(Z))].
\end{equation}
Since the intersection $Z \cap f(Z)$ is clean, it is a submanifold in $Y$;
hence, since $f$ is a diffeomorphism, $f^{-1}(Z \cap f(Z))$ is a submanifold in $Y$ as well.
Using~\eqref{eq:graphfZZRepresent} we deduce from this that $\graph f \cap Z \times Z$
is a submanifold in $Y \times Y$, as claimed.

2) Let $\pt \in Z$ be a fixed point such that $f(v) \in Z$.
We claim that the following holds:
\begin{equation}\label{eq:graphfZZtangent}
\Tpt{(\graph f \cap Z \times Z)}{f(\pt) \times \pt} =
\Tpt{(\graph f)}{f(\pt) \times \pt} \cap \Tpt{(Z \times Z)}{f(\pt) \times \pt}.
\end{equation}
Indeed, firstly note that
$$
	\Tpt{(\graph f)}{f(\pt) \times \pt} = \graph df
$$
(by $df$ we denote the linear map $\Tpt{Y}{\pt} \ra \Tpt{Y}{f(\pt)}$
induced by $f$), so we have
\begin{multline}\notag
\Tpt{(\graph f)}{f(\pt) \times \pt} \,\cap\, \Tpt{(Z \times Z)}{f(\pt) \times \pt} =
\graph df \,\cap\, (\Tpt{Z}{f(\pt)} \times \Tpt{Z}{\pt}) = \\
= (df,\id)\,[(df)^{-1}\,(\Tpt{Z}{f(\pt)} \cap df(\Tpt{Z}{\pt}))].
\end{multline}
(The last equation is analogous to~\eqref{eq:graphfZZRepresent}.)
Secondly, since the intersection $Z \cap f(Z)$ is clean, we have
$$
\Tpt{Z}{f(\pt)} \cap df(\Tpt{Z}{\pt}) = \Tpt{Z}{f(\pt)} \cap \Tpt{(f(Z))}{f(\pt)} =
\Tpt{(Z \cap f(Z))}{f(\pt)}.
$$
Therefore
\begin{multline}\notag
\Tpt{(\graph f)}{f(\pt) \times \pt} \,\cap\, \Tpt{(Z \times Z)}{f(\pt) \times \pt}
= (df,\id)\,[(df)^{-1}\,(\Tpt{Z}{f(\pt)} \cap df(\Tpt{Z}{\pt}))] = \\
= (df,\id)\,[(df)^{-1}\,(\Tpt{(Z \cap f(Z)}{f(\pt)})]
= d(f,\id)\,[\Tpt{(f^{-1}(Z \cap f(Z)))}{\pt}] = \\
= \Tpt{((f,\id)\,[f^{-1}(Z \cap f(Z))])}{f(\pt) \times \pt}
= \Tpt{(\graph f \cap Z \times Z)}{f(\pt) \times \pt}.
\end{multline}
(For the last equality we have used~\eqref{eq:graphfZZRepresent} directly.)
Thus we have got~\eqref{eq:graphfZZtangent}.

Lemma~\ref{lemma:graphCleanIntersection} is proved.
\end{proof}
Setting $Y = \TzM$, $Z = \TzMX$, $f = g$
and applying Lemma~\ref{lemma:graphCleanIntersection}, we see
that the condition 1) of Theorem~\ref{th:niceSled} is satisfied.

\textbf{Step 2.} \textit{Condition 2).}
Let us check that 
$\graph g \cap N^*(X\times X) = \emptyset$.
We use our abstract notation again.
\begin{lemma}\label{lemma:graphEmptyIntersection}
Let~$f\colon Y \ra Y$ is a diffeomorphism of smooth manifolds
and let~$Z \subset Y$ be a submanifold.
If $Z \cap f(Z) = \emptyset$ then $\graph f \cap Z\times Z = \emptyset$.
\end{lemma}
\begin{proof}
Obviously follows from~\eqref{eq:graphfZZRepresent}.
\end{proof}
Setting $Y = \TzM$, $Z = N^*_0X$, $f = g$,
and applying Lemma~\ref{lemma:graphEmptyIntersection}, we see
that the condition 2) of Theorem~\ref{th:niceSled} is satisfied as well.

Now Theorem~\ref{th:niceSled} implies Corollary~\ref{cor:qnaturality}.
\end{proof}

\end{document}